\documentclass{amsart}

\usepackage[alphabetic,initials,nobysame]{amsrefs}
\usepackage{amssymb,mathtools,mathrsfs,comment,todonotes,color}
\usepackage{hyperref} 
\usepackage[shortlabels]{enumitem}
\usepackage{caption}
\usepackage[percent]{overpic}

\BibSpec{collection.article}{%
	+{}  {\PrintAuthors}				{author}
	+{,} { \textit}                     {title}
	+{.} { }                            {part}
	+{:} { \textit}                     {subtitle}
	+{,} { \PrintContributions}         {contribution}
	+{,} { \PrintConference}            {conference}
	+{}  {\PrintBook}                   {book}
	+{,} { }                            {booktitle}
	+{,} { }							{series}
	+{,} { }							{publisher}
	+{,} { }							{address}
	+{,} { \PrintDateB}                 {date}
	+{,} { pp.~}                        {pages}
	+{,} { }                            {status}
	+{,} { \PrintDOI}                   {doi}
	+{,} { available at \eprint}        {eprint}
	+{}  { \parenthesize}               {language}
	+{}  { \PrintTranslation}           {translation}
	+{;} { \PrintReprint}               {reprint}
	+{.} { }                            {note}
	+{.} {}                             {transition}
	+{}  {\SentenceSpace \PrintReviews} {review}
}

\theoremstyle{plain}
\newtheorem{theorem}{Theorem}
\newtheorem{lemma}[theorem]{Lemma}

\newtheorem{corollary}[theorem]{Corollary}

\theoremstyle{definition}
\newtheorem{definition}[theorem]{Definition}

\theoremstyle{remark}

\numberwithin{equation}{section}
\numberwithin{theorem}{section}
\numberwithin{conjecture}{section}

\newcommand{\x}{\scalebox{1.2}{$\chi$} } 
\newcommand{\br}{\overline}
\newcommand{\R}{\mathbb R}

\newcommand{\C}{\mathbb C}

\newcommand{\Z}{\mathbb Z}
\newcommand{\N}{\mathbb N}

\newcommand{\h}{\mathscr H}

\DeclareMathOperator{\dist}{{\mathrm{dist}}}
\DeclareMathOperator{\diam}{{\mathrm{diam}}}

\DeclareMathOperator{\md}{\mathrm{Mod}}

\DeclareMathOperator{\loc}{\mathrm{loc}}

\DeclareMathOperator{\NED}{\mathit{NED}}
\DeclareMathOperator{\CNED}{\mathit{CNED}}

\begin{document}
\title{Metric definition of quasiconformality and exceptional sets}

\author{Dimitrios Ntalampekos}
\address{Mathematics Department, Stony Brook University, Stony Brook, NY 11794, USA.}

\thanks{The author was partially supported by NSF Grants DMS-2000096 and DMS-2246485, and by the Simons Foundation.}
\email{dimitrios.ntalampekos@stonybrook.edu}

\date{\today}

\begin{abstract}
We show that a homeomorphism of Euclidean space is quasiconformal if and only if at each point there exists a sequence of uncentered open sets with bounded eccentricity shrinking to that point whose images also have bounded eccentricity. This generalizes the metric definition of quasiconformality of Gehring that uses balls instead. We also study exceptional sets for this definition, in connection with sets that are negligible for extremal distances. We introduce the class of $\CNED$ sets, generalizing the classical notion of $\NED$ sets studied by Ahlfors--Beurling. A set $A$ is $\CNED$ if the conformal modulus of a curve family is not affected when one restricts to the subfamily intersecting the set $A$ at countably many points.  We show as our main theorem that $\CNED$ sets are exceptional for the definition of quasiconformality. 
\end{abstract}

\maketitle

\section{Introduction}

\subsection{A new definition of quasiconformality}
We assume throughout that $n\geq 2$. Let $\Omega\subset \R^n$ be an open set and $f\colon \Omega\to \R^n$ be a topological embedding, i.e., a homeomorphism onto its image, that is orientation-preserving.  We say that $f$ is quasiconformal if  $f\in W^{1,n}_{\loc}(\Omega)$ and  there exists $K\geq 1$ such that
\begin{align*}
\|Df(x)\|^n \leq KJ_f(x)
\end{align*}
for a.e.\ $x\in \Omega$. In this case, we say that $f$ is $K$-quasiconformal. This is known as the \textit{analytic} definition of quasiconformality. In what follows all topological embeddings are implicitly assumed to be orientation-preserving.

We define the \textit{metric distortion} of $f$ at a point $x\in \Omega$ by  
\begin{align}\label{intro:definition_metric}
H_f(x)= \limsup_{r\to 0} \frac{L_f(x,r)}{l_f(x,r)},
\end{align}
where, for $r>0$,
\begin{align*}
L_f(x,r)&=\sup\{|f(x)-f(y)|: y\in \Omega,\,\, |x-y|\leq r\}\quad \textrm{and}\\
l_f(x,r)&=\inf\{|f(x)-f(y)|: y\in \Omega,\,\, |x-y|\geq r\}.
\end{align*}
By a result of Gehring \cite{Gehring:Rings}*{Corollary 3},  $f$ is quasiconformal if and only if there exists $H\geq 1$ such that $H_f(x)\leq H$ for every $x\in \Omega$. This is known as the \textit{metric} definition of quasiconformality.  Geometrically, it says that $f$ maps \textit{all} sufficiently small balls centered at $x$ to topological balls with bounded \textit{eccentricity}. The eccentricity of a bounded open set $A\subset \R^n$ is by definition 
$$E(A)= \inf\{M\geq 1: \textrm{there exists an open ball}\,\, B \,\,\textrm{such that}\,\, B\subset A\subset MB \}.$$
Observe that the eccentricity of a ball is $1$ and if $B(x,r)\subset \br B(x,r)\subset  \Omega$, then 
\begin{align*}
E( f(B(x,r))) \leq \frac{L_f(x,r)}{l_f(x,r)}.
\end{align*} 
The reverse inequality is not true in general. If $f$ is quasiconformal then  
$$\limsup_{r\to 0}E( f(B(x,r)))$$
is uniformly bounded in $\Omega$.

A fundamental theorem proved by Heinonen--Koskela \cite{HeinonenKoskela:liminf} is that the ``limsup" in the definition of $H_f$ in \eqref{intro:definition_metric} can be replaced by ``liminf". Thus, only \textit{a sequence} of balls centered at $x$ and shrinking to $x$ is required to be mapped under $f$ to sets with bounded eccentricity. This  significant result was immediately applied in rigidity problems in complex dynamics in the work of Przytycki--Rohde \cite{PrzytyckiRohde:ColletEckmann} and in further works that we mention below. 

One natural question is whether one can define quasiconformality by requiring the symmetric condition that arbitrary sets of bounded eccentricity and not necessarily balls are mapped to sets of bounded eccentricity. We prove here that this is indeed the case. We first provide a definition. 

\begin{definition}\label{definition:eccentric}
Let $\Omega\subset \R^n$ be an open set and $f\colon \Omega \to \R^n$ be a topological embedding. The \textit{eccentric distortion of $f$} at a point $x\in \Omega$ is defined by
\begin{align*}
E_f(x) =\inf \{ M\geq 1:\,\, &\textrm{there exists a sequence of open sets $A_k\subset \Omega$, $k\in \N$,} \\
&\textrm{with $x\in A_k$, $k\in \N$, and $\diam(A_k)\to 0$ as $k\to\infty$}\\ &\textrm{such that $E(A_k)\leq M$ and $E(f(A_k))\leq M$ for all $k\in \N$}\},
\end{align*}
\end{definition}

Observe that 
$$E_f(x)\leq H_f(x)$$
for each $x\in \Omega$, thus, quasiconformal maps have uniformly bounded eccentric distortion. We prove that the converse is true. 

\begin{theorem}\label{theorem:eccentric}
Let $\Omega\subset \R^n$ be an open set and $f\colon \Omega \to \R^n$ be a topological embedding. Suppose that there exists a constant $H\geq 1$ such that for all $x\in \Omega$ we have
\begin{align}\label{theorem:eccentric:condition}
E_f(x)\leq H.
\end{align}
Then $f$ is quasiconformal in $\Omega$. 
\end{theorem}

Equivalently, $f$ is quasiconformal if for each $x\in \Omega$ there exists a sequence of open sets $A_k$, $k\in \N$, containing $x$ and shrinking to $x$ such that $A_k$ and $f(A_k)$ have uniformly bounded eccentricity, not depending on $k$ or $x$. One advantage of this condition, compared to the classical metric definition, is that it is completely symmetric with respect to $f$ and $f^{-1}$: 
\begin{align*}
E_f(x) = E_{f^{-1}}(f(x)).
\end{align*}
Another advantage is that the sets $A_k$ shrinking to $x$ are \textit{uncentered}, as opposed to the balls in the metric definition. This feature makes Theorem \ref{theorem:eccentric} very powerful, as illustrated by a compelling application in the problem of rigidity of circle domains that we discuss below. 

The proof of the theorem of Heinonen--Koskela, replacing ``limsup" with ``liminf" in \eqref{intro:definition_metric}, cannot be used for the proof of Theorem \ref{theorem:eccentric}. The reason is that it relies crucially on the Besicovitch covering theorem (see Theorem \ref{theorem:besicovitch}), which roughly asserts that a cover by open balls can be replaced by a subcover that has bounded multiplicity. This powerful tool can be used only for coverings by geometric balls and not by arbitrary sets of bounded eccentricity. Thus, for the proof of Theorem \ref{theorem:eccentric} we need a new technical covering lemma, which is one of the innovations of the current work and we term \textit{the egg-yolk covering lemma}.  We present this lemma in Section \ref{section:egg}. Theorem \ref{theorem:eccentric} is a special case of the more general Theorem \ref{theorem:main}, in which we allow for some exceptional sets as well, instead of requiring \eqref{theorem:eccentric:condition} at all points. 

\subsection{Exceptional sets for the definition of quasiconformality}

By a result of Gehring \cite{Gehring:Rings}*{Theorem 8}, in order to establish quasiconformality one does not need to verify condition \eqref{intro:definition_metric} at all points $x\in \Omega$, but can allow for some exceptional sets: a set of $\sigma$-finite Hausdorff $(n-1)$-measure, where we could have $H_f=\infty$, and a set of $n$-measure zero, where $H_f$ could be finite but unbounded.  On the other hand, the result and methods of Heinonen--Koskela \cite{HeinonenKoskela:liminf} do not allow for an exceptional set, if one replaces ``limsup" with ``liminf" in \eqref{intro:definition_metric}. 

Later Kallunki--Koskela \cites{KallunkiKoskela:quasiconformal,KallunkiKoskela:exceptional2} proved a significant generalization of the theorems of Gehring and Heinonen--Koskela, replacing ``limsup" with ``liminf" in \eqref{intro:definition_metric} and allowing for the same type of exceptional sets as Gehring's theorem. The possibility of an exceptional set in the Heinonen--Koskela theorem was immediately exploited for resolving rigidity problems in complex dynamics  \cites{GraczykSmirnov:rigidity,Haissinsky:rigidity,KozlovskiShenVStiren:rigidity,Smania:rigidity}.

There has been a long line of research in obtaining such  results for Sobolev functions in Euclidean space and for Sobolev and quasiconformal maps in metric spaces; see  \cites{HeinonenKoskela:qc, BaloghKoskela:Loewner,KoskelaShanmugalingamTyson:removable,KallunkiMartio:acl,KoskelaRogovin:acl ,BaloghKoskelaRogovin:qc,Williams:dilatation}.

Our main result, Theorem \ref{theorem:main}, is a further generalization of above results and allows for a much larger class of exceptional sets than sets of $\sigma$-finite Hausdorff $(n-1)$-measure. Namely, sets that are ``negligible for extremal distances" in some weak sense are exceptional for the definition of quasiconformality. We introduce some terminology before stating the result.

For an open set $U\subset \R^n$ and two continua  {$F_1,F_2\subset  U$} the family of curves joining $F_1$ and $F_2$ inside $U$ is denoted by $\Gamma(F_1,F_2;U)$. For a set $A\subset \R^n$ we denote by $\mathcal F_0(A)$ the family of curves in $\R^n$ that do not intersect $A$, {except possibly at the endpoints}, and by $\mathcal F_{\sigma}(A)$ the family of curves in $\R^n$ that intersect $A$  at countably many points, not counting multiplicity. 

A set $A\subset \R^n$ is \textit{negligible for extremal distances} if for every pair of non-empty, disjoint continua $F_1,F_2\subset \R^n$ we have
\begin{align*}
\md_n \Gamma(F_1,F_2;\R^n) =\md_n (\Gamma(F_1,F_2;\R^n)\cap \mathcal F_0(A)).
\end{align*}
In this case, we write $A\in \NED$. We remark that we do not require $A$ to be closed. Closed $\NED$ sets in the plane were studied and characterized in the seminal work of Ahlfors--Beurling \cite{AhlforsBeurling:Nullsets}. Specifically, a closed set $A$ is $\NED$ if and only if every conformal embedding $f\colon \C\setminus A\to \C$ is the restriction of a linear map. The role of $\NED$ sets in higher dimensions and their connection to removable sets for Sobolev functions were studied in \cites{Vaisala:null,AseevSycev:removable,VodopjanovGoldstein:removable}.

We introduce in this paper a significantly larger class of sets and show that they are exceptional for the definition of quasiconformality. We say that a set $A\subset \R^n$ is \textit{countably negligible for extremal distances} if 
\begin{align*}
\md_n \Gamma(F_1,F_2;\R^n) =\md_n (\Gamma(F_1,F_2;\R^n)\cap \mathcal F_{\sigma}(A))
\end{align*}
for every pair of non-empty, disjoint continua $F_1,F_2\subset \R^n$. In this case we write $A\in \CNED$.  Again, the set $A$ need not be closed. The monotonicity of modulus implies that $\NED\subset \CNED$.  We now state our main theorem.

\begin{theorem}\label{theorem:main}
Let $\Omega\subset \R^n$ be an open set and $f\colon \Omega \to \R^n$ be a topological embedding. Let $A,G\subset \Omega$ be sets such that 
\begin{align*}
A\in \CNED \,\,\, \textrm{and}\,\,\, \textrm{either} \,\,\, m_n(G)=0 \,\,\, \textrm{or}\,\,\, m_n(f(G))=0.
\end{align*}
Suppose that there exists a constant $H\geq 1$ such that for all $x\in \Omega\setminus (A\cup G)$ we have
\begin{align*}
 E_f(x) \leq H,
\end{align*}
and for all $x\in G$ we have
\begin{align*}
 E_f(x)<\infty.
\end{align*}
Then $f$ is $K$-quasiconformal in $\Omega$, for some $K\geq 1$ depending only on $n$ and $H$.
\end{theorem}

Here $m_n$ denotes the $n$-dimensional Lebesgue measure. If $A=G=\emptyset$, then we obtain Theorem \ref{theorem:eccentric}. The proof of Theorem \ref{theorem:main} is presented in Section \ref{section:proof}. The central technical device for the proof is Theorem \ref{theorem:upper_gradient}.

We remark that Theorem \ref{theorem:main} is innovative in three different directions, compared to previous results of Gehring, Heinonen--Koskela, and Kallunki--Koskela. First, we assume upper bounds for the eccentric distortion $E_f$ rather than the metric distortion $H_f$; recall that $E_f\leq H_f$. Second, our proof gives a new perspective and allows the possibility that either $m_n(G)=0$ or $m_n(f(G))=0$, while in previous works only the first assumption was considered. Third, the set $A$ is assumed to be $\CNED$, while in the past only sets of $\sigma$-finite Hausdorff $(n-1)$-measure were considered. In \cite{Ntalampekos:cned} the current author shows that the class of $\CNED$ sets includes sets of $\sigma$-finite Hausdorff $(n-1)$-measure, as well as, many other known classes of \textit{quasiconformally removable} sets. A closed set $A\subset \R^n$ is quasiconformally removable if every homeomorphism of $\R^n$ that is quasiconformal in $\R^n\setminus A$ is quasiconformal in $\R^n$.  Thus, we have the following consequence of Theorem \ref{theorem:main}.
\begin{corollary}
Closed $\CNED$ sets are quasiconformally removable. 
\end{corollary}
It is an open problem to characterize such sets even in dimension $2$. Known classes of removable sets include sets of $\sigma$-finite Hausdorff $(n-1)$-measure \cites{Besicovitch:Removable, Gehring:Rings}, sets with good geometry, such as boundaries of John and H\"older domains \cites{Jones:removability,JonesSmirnov:removability}, and $\NED$ sets \cite{AhlforsBeurling:Nullsets}. In the subsequent work \cite{Ntalampekos:cned} the current author shows that the above-mentioned classes of sets are also in the $\CNED$ class, suggesting that closed $\CNED$ sets characterize quasiconformally removable 
sets. 

Theorem \ref{theorem:main} has already found an application in the deep problem of rigidity of circle domains. A circle domain in the plane is \textit{conformally rigid} if every conformal map onto another circle domain is the restriction of a M\"obius transformation. It is conjectured by He--Schramm \cite{HeSchramm:Rigidity} that a circle domain is rigid if and only if its boundary is quasiconformally removable. The conjecture has been established in some cases by He--Schramm and by the author in joint work with Younsi \cite{NtalampekosYounsi:rigidity}. With the aid of Theorem \ref{theorem:main} the current author \cite{Ntalampekos:rigidity_cned} is able to establish that circle domains with $\CNED$ boundary are rigid,  a result that features not only $\CNED$ sets, but also the use of the eccentric distortion in the definition of quasiconformality. This development is the strongest so far and provides substantial evidence for the conjecture of He--Schramm; for example, if one can show that $\CNED$ sets coincide with quasiconformally removable sets, then the conjecture is true for domains with totally disconnected boundary as a consequence of \cite{Ntalampekos:rigidity_cned}. 

We expect that our results will find further applications in rigidity problems in complex dynamics, where often one has no geometric information about the distortion of balls, but can control the distortion of non-round dynamical objects, such as \textit{puzzle pieces}.

\subsection*{Acknowledgments} The author would like to thank the referee for carefully reading the paper and providing valuable comments and suggestions.

\section{The egg-yolk covering lemma}\label{section:egg}

For quantities $A$ and $B$ we write $A\lesssim B$ if there exists a constant $c>0$ such that $A\leq cB$. If the constant $c$ depends on another quantity $H$ that we wish to emphasize, then we write instead $A\leq c(H)B$ or $A\lesssim_H B$. Moreover, we use the notation $A\simeq B$ if $A\lesssim B$ and $B\lesssim A$. As previously, we write $A\simeq_H B$ to emphasize the dependence of the implicit constants on the quantity $H$. All constants in the statements are assumed to be positive even if this is not stated explicitly and the same letter may be used in different statements to denote a different constant.  

\subsection{Known covering results}
We first state a classical covering result.

\begin{lemma}[$5B$-covering lemma, \cite{Heinonen:metric}*{Theorem 1.2, p.~2}]\label{lemma:5b}
Let $X$ be a metric space and $\mathcal B$ be a collection of open balls in $X$ with uniformly bounded radii. Then there exists a disjointed subcollection $\mathcal B'$ of $\mathcal B$ such that
$$ \bigcup_{B\in \mathcal B} B \subset \bigcup_{B\in \mathcal B'}5B.$$
\end{lemma}
For an open ball $B=B(x_0,r_0)$ and $\lambda>0$ we denote by $\lambda B$ the ball $B(x_0,\lambda r_0)$. Note that in metric spaces the center and radius of a ball need not be unique, so we regard the ball $B(x_0,r_0)$ not only as a set, but also as a pair $(x_0,r_0)$. Then there is no ambiguity in the definition of $\lambda B$.  

The power of the $5B$-covering lemma lies on the fact that it allows us to replace arbitrary covers by balls with covers by essentially disjoint balls. One drawback of the $5B$-covering lemma, however, is that if $f$ is an arbitrary homeomorphism on $X$, then there is no relation between the sizes of $f(B)$ and $f(5B)$. In particular, rescaling the family $\{f(5B)\}_{B\in \mathcal B'}$ by a uniform fixed factor will not give a disjointed family in general; more specifically, one cannot find a scaling factor $\lambda \in (0,1)$ and points $x_B\in f(5B)$ so that the family $\{B(x_B, \lambda \diam(f(5B)))\}_{B\in \mathcal B'}$ is disjointed. For this reason, when working with homeomorphisms of \textit{Euclidean} space, one can instead use the Besicovitch covering theorem.
\begin{theorem}[Besicovitch covering theorem, \cite{Mattila:geometry}*{Theorem 2.7}]\label{theorem:besicovitch}
Let $A\subset \R^n$ be a bounded set and $\mathcal B$ be a family of closed balls such that each point of $A$ is the center of a ball in $\mathcal B$. Then there exists a subcollection $\mathcal B'$ of $\mathcal B$ such that 
\begin{align*}
A\subset \bigcup_{B\in \mathcal B'} B\quad \textrm{and} \quad \sum_{B\in \mathcal B'} \x_{B} \leq c(n).
\end{align*}
\end{theorem}
Obviously, if $f$ is a homeomorphism of $\R^n$, then the family $\{f(B)\}_{B\in \mathcal B'}$ covers the set $f(A)$ with uniformly bounded multiplicity. Hence, unlike the $5B$-covering lemma, here we obtain information for both $\{B\}_{B\in \mathcal B'}$ and $\{f(B)\}_{B\in \mathcal B'}$. The drawback of this theorem is that it only works with \textit{geometric} balls in \textit{Euclidean} space and there is no generalization for covers by sets of bounded eccentricity, as defined in the Introduction, or for balls in metric spaces.

The egg-yolk covering lemma that we prove in this section can be regarded as a generalization of the $5B$-covering lemma and the Besicovitch covering theorem, giving favorable covers that encode geometric information both in the domain and the range of a homeomorphism between metric spaces.

Before moving to the statement of the egg-yolk covering lemma, we state a well-known inequality that is often used in combination with covering lemmas.

\begin{lemma}[\cite{Bojarski:inequality}]\label{lemma:bojarski}
Let $p\geq 1$ and $\lambda >0$. Suppose that $\{B_i\}_{i\in \N}$ is a collection of balls in $\R^n$ and $a_i$, $i\in \N$, is a sequence of non-negative numbers. Then
\begin{align*}
\left \|  \sum_{i\in \N} a_i \x_{\lambda B_i} \right \|_{L^p(\R^n)} \leq c({n,p,\lambda})  \left \|  \sum_{i\in \N} a_i \x_{B_i} \right \|_{L^p(\R^n)}.
\end{align*} 
\end{lemma}

\subsection{Egg-yolk pairs}
Let $(X,d)$ be a connected metric space. For a ball $B=B(x_0,r_0)\subset X$, we define $r(B)=r_0$. We always have $\diam(B)\leq 2r(B)$ and if $X\setminus B\neq \emptyset$, since $X$ is connected, we have
$$r(B)\leq \diam(B)\leq 2r(B).$$
Let $A\subset X$ be a bounded open set and $M\geq 2$. Suppose that there exists an open ball $B=B(x_0,r_0)$ such that $B\subset 2B\subset A\subset MB$. Then we call $(A,B)$ an \textit{$M$-egg-yolk pair}; see Figure \ref{figure:eggyolk}. If $(A,B)$ is an $M$-egg-yolk pair, we have the following immediate properties.
\begin{enumerate}[\upshape(EY1)]
\item $\diam(A)\leq 2Mr(B)$.
\item\label{egg:comparable} If $X\setminus A\neq \emptyset$, then $$\diam(B)\leq 2r(B)\leq \diam(2B)\leq \diam(A)\leq 2Mr(B)\leq 2M\diam(B).$$
\item\label{egg:distance} If $X\setminus A\neq \emptyset$, then $\dist(B,X\setminus A)\geq r(B)$.
\item\label{egg:ball_closure} If $x\in B$ and $y\in \br A$, then $d(x,y) \leq(M+1)r(B)$. 
\end{enumerate}
Moreover, the following statements are true.
\begin{enumerate}[\upshape(EY1)]\setcounter{enumi}{4}
\item\label{egg:intersect_yolk} Let $(A_i,B_i)$ be $M$-egg-yolk pairs, for $i=1,2$, such that $B_1\cap B_2\neq \emptyset$ and $\br{A_2}\not\subset A_1$. Then
\begin{align*}
 {\diam(A_2)} \geq c(M) \diam(A_1).
\end{align*} 
\end{enumerate}
\begin{proof}[Proof of \ref{egg:intersect_yolk}]
If $A_2=X$ there is nothing to prove, so we assume that $X\setminus A_2\neq \emptyset$. Also, since $\br{A_2}\not\subset A_1$, we cannot have $A_1=X$; thus $X\setminus A_1\neq \emptyset$. If $(M+1)r(B_2)<\dist(B_1, X\setminus A_1)$, then by \ref{egg:ball_closure}  for $x\in B_1\cap B_2$ and $y\in \br {A_2}$ we have $d(x,y)\leq (M+1)r(B_2)<\dist(B_1,X\setminus A_1)$. Thus, by the triangle inequality,
$$\dist(y,X\setminus A_1) \geq \dist(B_1,X\setminus A_1)-d(x,y)>0.$$
It follows that $\br{A_2}\subset A_1$, a contradiction. Therefore, by \ref{egg:distance} and \ref{egg:comparable},
$$(M+1)r(B_2)\geq \dist(B_1, X\setminus A_1)\geq r(B_1) \geq 2^{-1}M^{-1}\diam(A_1)$$
and, by \ref{egg:comparable} again,
\[\diam(A_2) \geq 2r(B_2)\geq \frac{1}{M(M+1)}\diam(A_1). \qedhere\]
\end{proof}
\begin{enumerate}[\upshape(EY1)]\setcounter{enumi}{5}
\item\label{egg:cluster} Let $(A_i,B_i)$, $i\in I$, be a family of $M$-egg-yolk pairs and suppose that there exists $i_0\in I$ such that $A_i\cap A_{i_0}\neq \emptyset$ and $\diam(A_i)\leq a \diam(A_{i_0})$ for each $i\in I$ and for some $a>0$. We set $A_I=\bigcup_{i\in I}A_i$. Then $(A_I,B_{i_0})$ is a $c(a,M)$-egg-yolk pair. 
\end{enumerate}
\begin{proof}[Proof of \ref{egg:cluster}]
Note that $B_{i_0}\subset 2B_{i_0} \subset A_{i_0}\subset A_I$ and $A_i\subset (2a+1)M B_{i_0}$ for each $i\in I$. Thus, $A_I\subset (2a+1)MB_{i_0}$.
\end{proof}

\begin{figure}
\centering
	\begin{overpic}[width=.4\linewidth]{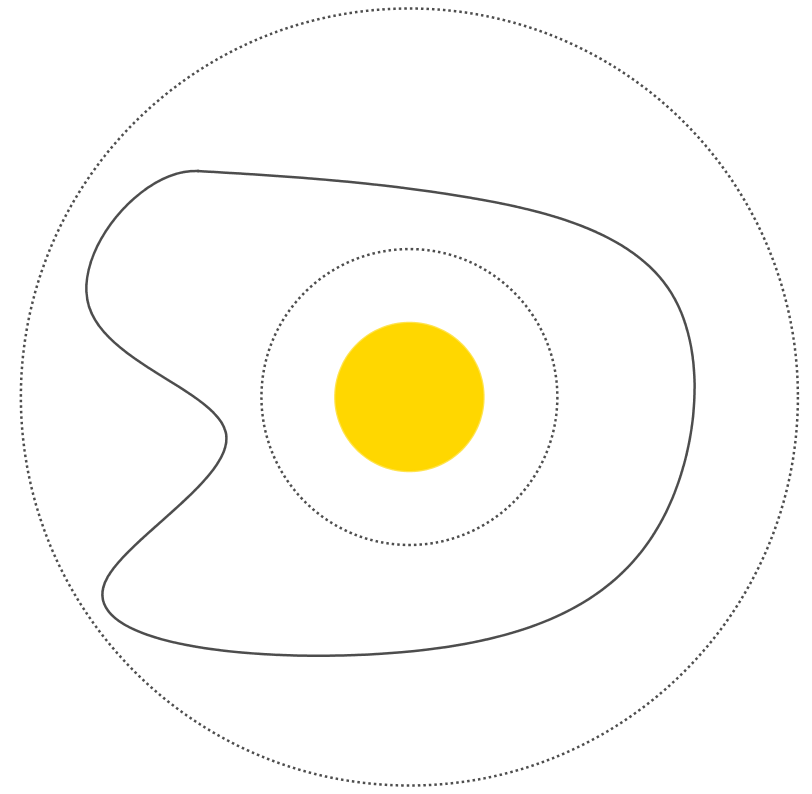}
		\put (50,50) {$B$}
		\put (66,60) {$2B$}
		\put (75,70) {$A$}
		\put (88,82) {$MB$}		
	\end{overpic}
	\caption{An $M$-egg-yolk pair.}\label{figure:eggyolk}
\end{figure}
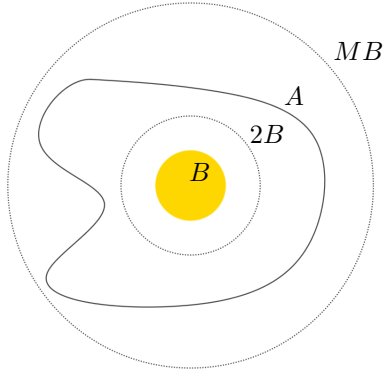

\subsection{The egg-yolk covering lemma}

\begin{lemma}[Egg-yolk covering lemma]\label{lemma:eggyolk}
Let $X,Y$ be compact, connected metric spaces,  $f\colon X\to Y$ be a homeomorphism, and $M\geq 2$. Let $\{(A_i,B_i)\}_{i\in I}$ and $\{(A_i',B_i')\}_{i\in I}$ be families of $M$-egg-yolk pairs in $X$ and $Y$, respectively, with $f(A_i)=A_i'$ for each $i\in I$. Then there exists a set $J\subset I$ and families $\{(D_j,B_j)\}_{j\in J}$ and $\{(D_j',B_j')\}_{j\in J}$ of $c(M)$-egg-yolk pairs in $X$ and $Y$, respectively, such that 
\begin{enumerate}[\upshape(i)]\smallskip
	\item\label{yolk:union} $\bigcup_{j\in J} D_j =\bigcup_{i\in I} A_i$,\smallskip
	\item\label{yolk:image} $f(D_j)=D_j'$ for each $j\in J$, and \smallskip
	\item\label{yolk:disjoint} the balls $B_j$, $j\in J$, are pairwise disjoint and the balls $B_j'$, $j\in J$, are pairwise disjoint.
\end{enumerate}
\end{lemma}

The remaining of the section is devoted to the proof of the egg-yolk covering lemma. The reader interested in the proof of the main theorem of the paper, Theorem \ref{theorem:main}, may skip the rest of Section \ref{section:egg} and proceed with Section \ref{section:proof}. 

It is crucial for the application of the lemma that we are not requiring $B_i$ to be related to $B_i'$; we are only assuming that $f(A_i)=A_i'$. In the case that $X=Y$, $f$ is the identity map, and $B_i'=B_i$, compare this lemma to the $5B$-covering lemma. 

The main idea of the proof is to create the sets $D_k$ with the aid of property \ref{egg:cluster}, by taking the union of sets $A_i, A_j$ whenever $B_i\cap B_j\neq \emptyset$. The essential difficulty is that the diameters of $A_i$ and $A_j$ might not be comparable. We first establish an auxiliary result, which allows us to reduce to the case that $A_i$ and $A_j$ have comparable diameters whenever $B_i\cap B_j\neq \emptyset$. 

\begin{lemma}\label{lemma:eggyolk_auxiliary}
Under the assumptions of Lemma \ref{lemma:eggyolk}, there exists a set $J\subset I$ and families  $\{(F_j, B_j)\}_{j\in J}$ and $\{(F_j', B_j')\}_{j\in J}$ of $c(M)$-egg-yolk pairs in $X$ and $Y$, respectively, such that 
\begin{enumerate}[\upshape(i)]\smallskip
	\item \label{egg:aux:i}$\bigcup_{j\in J} F_j=\bigcup_{i\in I}A_i$,\smallskip
	\item\label{egg:aux:ii} $f(F_j)=F_j'$ for each $j\in J$, and\smallskip
	\item\label{egg:aux:iii} if $B_i\cap B_j\neq \emptyset$ (resp.\ $B_i'\cap B_j'\neq \emptyset$) for some $i,j\in J$, then 
	\begin{align*}
	c(M)^{-1}\leq \frac{\diam(F_i)}{\diam(F_j)} \leq c(M) \quad \left(\textrm{resp.} \quad c(M)^{-1}\leq \frac{\diam(F_i')}{\diam(F_j')} \leq c(M)\right).
	\end{align*}
\end{enumerate}
\end{lemma}

We remark that $\{B_j\}_{j\in J}$ and $\{B_j'\}_{j\in J}$ are just subcollections of $\{B_i\}_{i\in I}$ and $\{B_i'\}_{i\in I}$, respectively, which are given in the assumptions of Lemma \ref{lemma:eggyolk}. 

If $I$ were a finite index set, then one could choose $\{F_j\}_{j\in J}$ to be a subcollection of $\{A_i\}_{i\in I}$ satisfying \ref{egg:aux:i} and with the property that $F_i\not\subset F_j$ whenever $i\neq j$. Then \ref{egg:intersect_yolk} would immediately give the crucial property \ref{egg:aux:iii} in Lemma \ref{lemma:eggyolk_auxiliary}. In the case that $I$ is infinite, the idea is the same, but the proof is more involved. 

\begin{proof}[Proof of Lemma \ref{lemma:eggyolk_auxiliary}]
Note that the collection $\{A_i\}_{i\in I}$ is partially ordered with respect to inclusion. By the Hausdorff maximal principle \cite{Munkres:topology}*{\S 1.11, p.~69}, for each $k\in I$ there exists a maximal totally ordered set $W(k)=\{A_{j}\}_{j\in J(k)} \subset \{A_i\}_{i\in I}$ containing $A_{k}$. Since $f$ is injective and  $f(A_i)=A_i'$, the collection $\{A_j'\}_{j\in J(k)}$ is also a maximal totally ordered subcollection of $\{A_i'\}_{i\in I}$. Define $A_{W(k)}= \bigcup_{j\in J(k)} A_j$ and $A'_{W(k)}=f(A_{W(k)})=\bigcup_{j\in J(k)}A_j'$. Obviously, $\bigcup_{k\in I}A_{W(k)}= \bigcup_{i\in I}A_i$. We define 
\begin{align*}
L(W(k))&= \sup \{\diam(A_j): A_j \in W(k) \}\quad \textrm{and} \\
L'(W(k))&=  \sup \{\diam(A_j'): A_j \in W(k) \}.
\end{align*}
Note that both numbers are finite since $A_j\subset X$, $A_j'\subset Y$, and $X,Y$ are bounded spaces. We fix $A_{i_1},A_{i_2}\in W(k)$ such that $\diam(A_{i_1})\geq L(W(k))/2$ and $\diam(A_{i_2}')\geq L'(W(k))/2$. Since $W(k)$ is totally ordered, we have $A_{i_1}\supset A_{i_2}$ or $A_{i_2}\supset A_{i_1}$. Without loss of generality, assume that $A_{i_1}\supset A_{i_2}$. Since $f(A_i)=A_i'$ for each $i\in I$, we have $A_{i_1}'\supset A_{i_2}'$. Thus, $\diam(A_{i_1}')\geq \diam(A_{i_2}') \geq L'(W(k))/2$.  Summarizing, there exists $A_{i(k)}\in W(k)$ such that $\diam(A_{i(k)})\geq L(W(k))/2$ and $\diam(A_{i(k)}')\geq L'(W(k))/2$.  Note that $\diam(A_j) \leq L(W(k)) \leq 2\diam(A_{i(k)})$ and $\diam(A_j') \leq L'(W(k)) \leq 2\diam(A_{i(k)}')$ for each $A_j\in W(k)$. Moreover, for each $A_j\in W(k)$, we have $A_j\cap A_{i(k)}\neq \emptyset$ and $A_{j}'\cap A_{i(k)}'\neq \emptyset$ by the total ordering of $W(k)$.  By property \ref{egg:cluster} we conclude that $(A_{W(k)},B_{i(k)})$ and $(A_{W(k)}',B_{i(k)}')$ are $c(M)$-egg-yolk pairs for each $k\in I$. 

If $A_{W(k)}=X$ for some $k\in I$, then we set $j=i(k)$, $J=\{j\}$,  $F_j=A_{W(k)}$, $F_j'=f(F_j)$, and we have nothing to prove. Hence, we suppose that $X\setminus A_{W(k)}\neq \emptyset$, and thus $X\setminus A_k\neq \emptyset$, for each $k\in I$.  

We claim that 
\begin{align}\label{lemma:aux:comparable}
\diam(A_{W(k)}) \simeq_M \diam(A_{W(l)}) \quad \textrm{whenever}\quad B_{i(k)}\cap B_{i(l)}\neq \emptyset.
\end{align}
The same is true for $(A_{W(k)}',B_{i(k)}')$, $k\in I$. To see this, suppose that $B_{i(k)}\cap B_{i(l)}\neq \emptyset$. If  $\br{A_{W(k)}} \subset A_{W(l)}=\bigcup_{j\in  J(l)}A_j$, then by the compactness of $\br{A_{W(k)}}$ and the total ordering of $W(l)$, there exists an open set $A_j \in W(l)$ such that $\br{A_{W(k)}}\subset A_j$. If $A_j\in W(k)$, then $\br{A_{W(k)}}=A_j$, so $A_j$ is clopen. By the connectedness of $X$, $A_j=X$, a contradiction. Therefore, $A_j\in W(l)\setminus W(k)$. This implies that $W(k)\cup \{A_j\}$ is totally ordered, which contradicts the maximality of $W(k)$. Therefore, $\br{A_{W(k)}} \not\subset A_{W(l)}$ and by \ref{egg:intersect_yolk} we have  $\diam(A_{W(k)}) \gtrsim_M \diam(A_{W(l)})$. By reversing the roles of $k$ and $l$, we see that $\diam(A_{W(k)}) \simeq_M \diam(A_{W(l)})$.

If the mapping $k\mapsto i(k)$ were injective on $I$, then the proof would have been completed with $J=i(I)\subset I$. In general, this might not be the case. For $j\in J=i(I)$, we define $F_j$ to be the union of all sets $A_{W(k)}$ such that $i(k)=j$.  Since $(A_{W(k)},B_j)$ is a $c(M)$-egg-yolk pair whenever $i(k)=j$, we conclude by \ref{egg:comparable} that 
$$\diam(A_{W(k)})\simeq_M \diam(B_j).$$
By property \ref{egg:cluster}, $(F_j,B_j)$ is a $c'(M)$-egg-yolk pair. We also set $F_j'=f(F_j)$ and similarly, $(F_j',B_j')$ is a $c'(M)$-egg-yolk pair. Without loss of generality, assume that $X\setminus F_j\neq \emptyset$ for each $j\in J$.  We only have to justify \ref{egg:aux:iii}. Suppose $B_{j_1}\cap B_{j_2}\neq \emptyset$ for some $j_1,j_2\in J$ and consider $k,l\in I$ with $i(k)=j_1$ and $i(l)=j_2$. Then, by \ref{egg:comparable} and \eqref{lemma:aux:comparable}, we have
\begin{align*}
\diam(F_{j_1})&\simeq_M \diam(B_{j_1})\simeq_M \diam(A_{W(k)})\simeq_M \diam(A_{W(l)})\\
&\simeq_M \diam(B_{j_2})\simeq_M\diam(F_{j_2}).
\end{align*}
The same argument applies to $(F_j',B_j')$, $j\in J$. This completes the proof.
\end{proof}

\begin{proof}[Proof of Lemma \ref{lemma:eggyolk}]
We will show that given $\{(A_i,B_i)\}_{i\in I}$ and $\{(A_i',B_i')\}_{i\in I}$ as in the statement, there exist families $\{(D_j,E_j)\}_{j\in J}$ and $\{(D_j',E_j')\}_{j\in J}$ of $c(M)$-egg-yolk pairs, where $\{E_j\}_{j\in J}$ and $\{E_j'\}_{j\in J}$ are subcollections of $\{B_i\}_{i\in I}$ and $\{B_i'\}_{i\in I}$, respectively,  satisfying conclusions \ref{yolk:union}, \ref{yolk:image}, and such that the sets $\{E_j'\}_{j\in J}$ are pairwise disjoint; that is, only one half of conclusion \ref{yolk:disjoint} is satisfied. Then using this statement for $f^{-1}$ and for the given $\{(D_j,E_j)\}_{j\in J}$ and $\{(D_j',E_j')\}_{j\in J}$ (in place of $\{(A_i,B_i)\}_{i\in I}$ and $\{(A_i',B_i')\}_{i\in I}$), we may find families $\{(\widetilde D_j,\widetilde E_j)\}_{j\in \widetilde J}$ and $\{(\widetilde D_j',\widetilde E_j')\}_{j\in \widetilde J}$ of $\widetilde c(M)$-egg-yolk pairs, satisfying the full conclusions of the lemma.

If $\diam(A_i)=0$ for some $i\in I$, then $A_i$ is a singleton and is clopen. The connectedness of $X$ implies that $X$ is a singleton. In this case there is nothing to prove, so we assume that $\diam(A_i)>0$ for each $i\in I$. 

By Lemma \ref{lemma:eggyolk_auxiliary}, we may assume that the given $\{(A_i,B_i)\}_{i\in I}$ and $\{(A_i',B_i')\}_{i\in I}$ are families of $c'(M)$-egg-yolk pairs with $f(A_i)=A_i'$ for each $i\in I$ and with the additional property that 
\begin{align}\label{yolk:reduction}
	\diam(A_i')&\simeq_M\diam(A_j') \quad \textrm{whenever}\quad B_i'\cap B_j'\neq \emptyset.
\end{align}

We set $L=\sup_{i\in I} \diam(A_i)$, which is positive and finite, since the space $X$ is bounded. Define $\mathcal F_0=\emptyset$ and $k_0=0$. Suppose that $\mathcal F_j\subset I$ and $k_j\in \Z$ have been defined for $j\in \{0,\dots,m\}$ such that $k_j$ is increasing in $j\in \{0,\dots,m\}$ and suppose that we have obtained $c(M)$-egg-yolk pairs $(D_i,E_i)$ and $(D_i',E_i')$ for $i\in \{1,\dots,k_m\}$, where $E_i\in \{B_j\}_{j\in I}$ and $E_i'\in \{B_j'\}_{j\in I}$, such that
\begin{enumerate}[\upshape(1),leftmargin=1.1cm]\smallskip
	\item \label{yolk:i:inclusion}
	$\mathcal F_0\cup \cdots\mathcal \cup \mathcal F_m\subset \{i\in I: \diam(A_i)>2^{-m}L\} \subset \{i\in I: A_i\subset D_j\,\, \textrm{for some}\,\, j\in \{1,\dots,k_m\} \}$ and $\bigcup_{i=1}^{k_m}D_i\subset \bigcup_{i\in I}A_i$,\smallskip
	\item\label{yolk:i:image} $f(D_i)=D_i'$, $i\in \{1,\dots, k_m\}$,\smallskip
	\item\label{yolk:i:disjoint} the sets $E_i'$, $i\in \{1,\dots, k_m\}$, are pairwise disjoint, and\smallskip
	\item\label{yolk:i:containment}  $\{ i\in I: B_i'\cap E_j'\neq \emptyset \,\, \textrm{for some}\,\, j\in \{1,\dots,k_{m}\}\}\subset \{i\in I: A_i\subset D_j\,\, \textrm{for some}\,\,\\ j\in \{1,\dots,k_m\} \}.$  
\end{enumerate}
Note that all of these statements are vacuously true for $m=0$.  Assuming that the above statements are true for each $m\in \N\cup\{0\}$, we see that  \ref{yolk:i:inclusion}, \ref{yolk:i:image}, and \ref{yolk:i:disjoint} give \ref{yolk:union}, \ref{yolk:image}, and \ref{yolk:disjoint}, respectively, completing the proof of the lemma.

Now we show the inductive step. We define   
$$\mathcal F_{m+1}= \{ i\in I : 2^{-m-1}L< \diam(A_i)\leq 2^{-m} L   \quad \textrm{and}\quad A_i\not\subset  D_j,\,\, j\in \{1,\dots,k_m\}\}.$$
If $\mathcal F_{m+1}=\emptyset$, we define $k_{m+1}=k_m$. Suppose that $\mathcal F_{m+1}\neq \emptyset$ and let $i_1\in \mathcal F_{m+1}$. Since $A_{i_1}\not\subset D_j$ for all $j\in \{1,\dots,k_m\}$, we conclude by the induction assumption \ref{yolk:i:containment} that $B_{i_1}' \cap E_j'=\emptyset$ for all $j\in \{1,\dots,k_m\}$. Suppose that $B_{i_1}'\cap B_j'\neq \emptyset $ for some $j\in I$ with $A_j\not\subset D_i$ for all $i\in \{1,\dots,k_m\}$. By \eqref{yolk:reduction} we conclude that $\diam(A_{i_1}')\simeq_M \diam(A_{j}')$. We define $E_{k_m+1}'=B_{i_1}'$ and $D_{k_m+1}'$ to be the union of $A_{i_1}'$ with the sets $A_j'$ such that $B_{i_1}'\cap B_j'\neq \emptyset$ and  $A_j\not\subset D_i$ for all $i\in \{1,\dots,k_m\}$; see Figure \ref{figure:eggyolk_cluster}. By \ref{egg:cluster}, we conclude that $(D_{k_m+1}',E_{k_m+1}')$ is a $c(M)$-egg-yolk pair. Define $D_{k_m+1}=f^{-1}(D_{k_m+1}')$ and $E_{k_m+1}=B_{i_1}$. Note that $D_{k_m+1}$ is the union of $A_{i_1}$ with some sets $A_j$ such that $A_{i_1}\cap A_j\neq \emptyset$ (since $A_{i_1}'\cap A_j'\supset B_{i_1}'\cap B_j'\neq\emptyset$) and $A_j\not\subset D_i$ for $i\in \{1,\dots,k_m\}$; thus, by the induction assumption \ref{yolk:i:inclusion} we have $\diam(A_j)\leq 2^{-m}L < 2\diam(A_{i_1})$. It follows that $(D_{k_m+1},E_{k_m+1})$ is a $c(M)$-egg-yolk pair by \ref{egg:cluster}. We remark that by construction we have $\{ i\in I: B_i'\cap E_{k_m+1}'\neq \emptyset  \}\subset \{i\in I: A_i\subset D_j\,\, \textrm{for some}\,\, j\in \{1,\dots,k_m+1\} \}.$

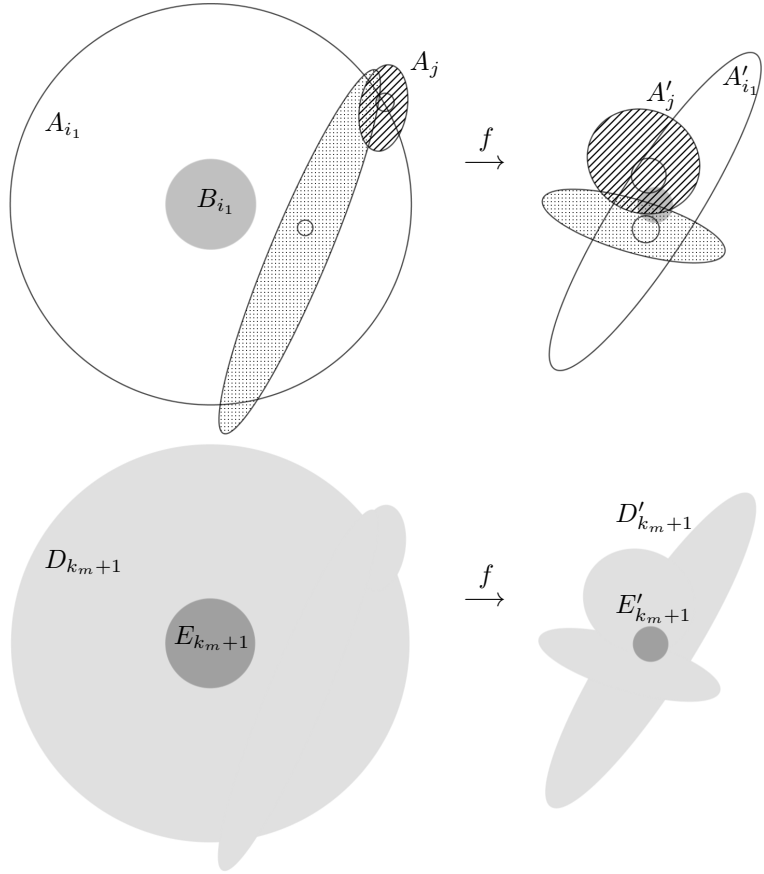
\begin{figure}
\centering
	\begin{overpic}[width=.8\linewidth]{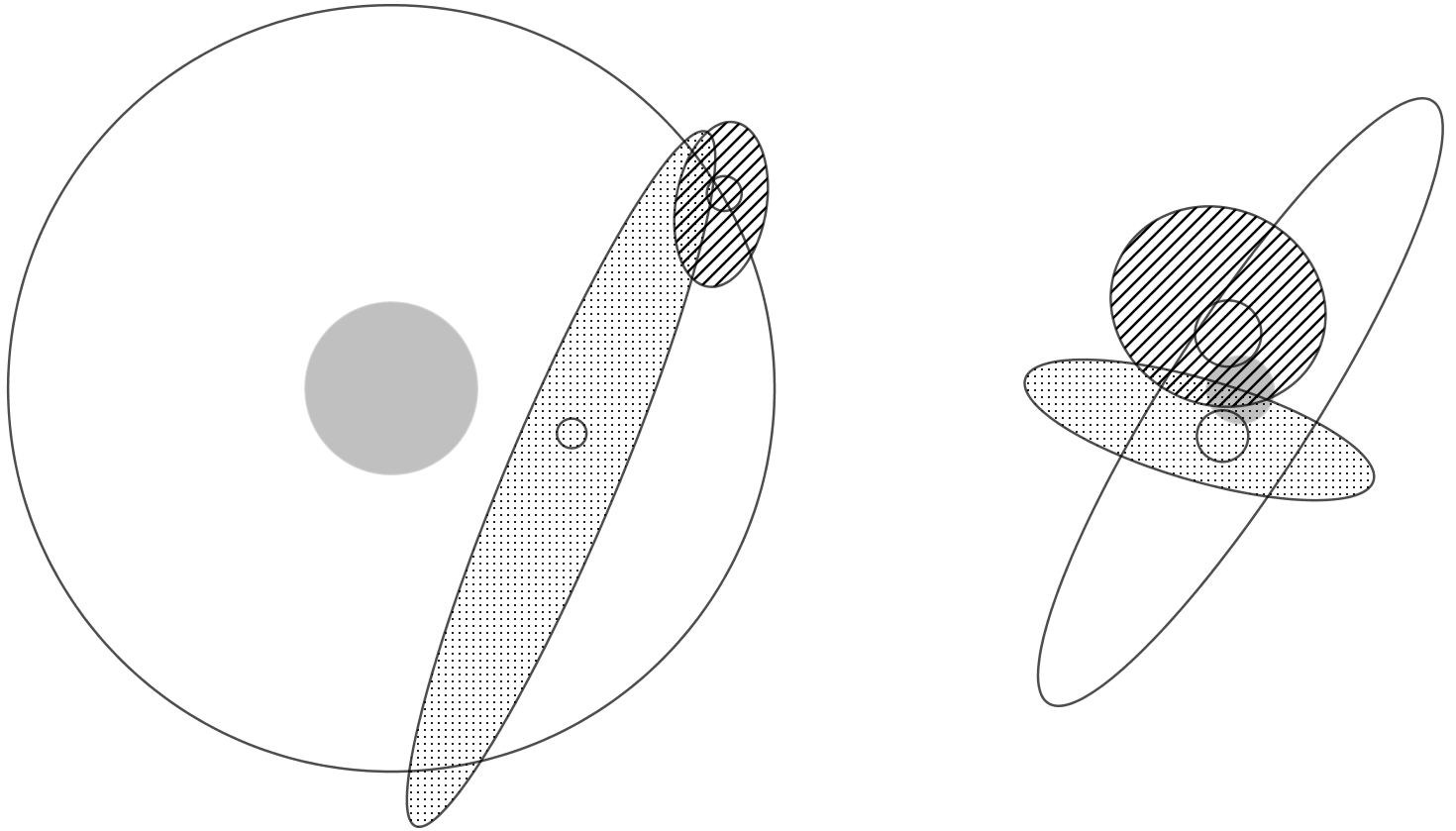}
		\put (25,30) {$B_{i_1}$}
		\put (5,40) {$A_{i_1}$}
		\put (94,46) {$A_{i_1}'$}
		\put (53,48) {$A_j$}
		\put (84,44) {$A_j'$}
		\put (60,35) {$\longrightarrow$}
		\put (62,38) {$f$}
	\end{overpic}
	
\vspace{0.3em}

	\begin{overpic}[ width=.8\linewidth]{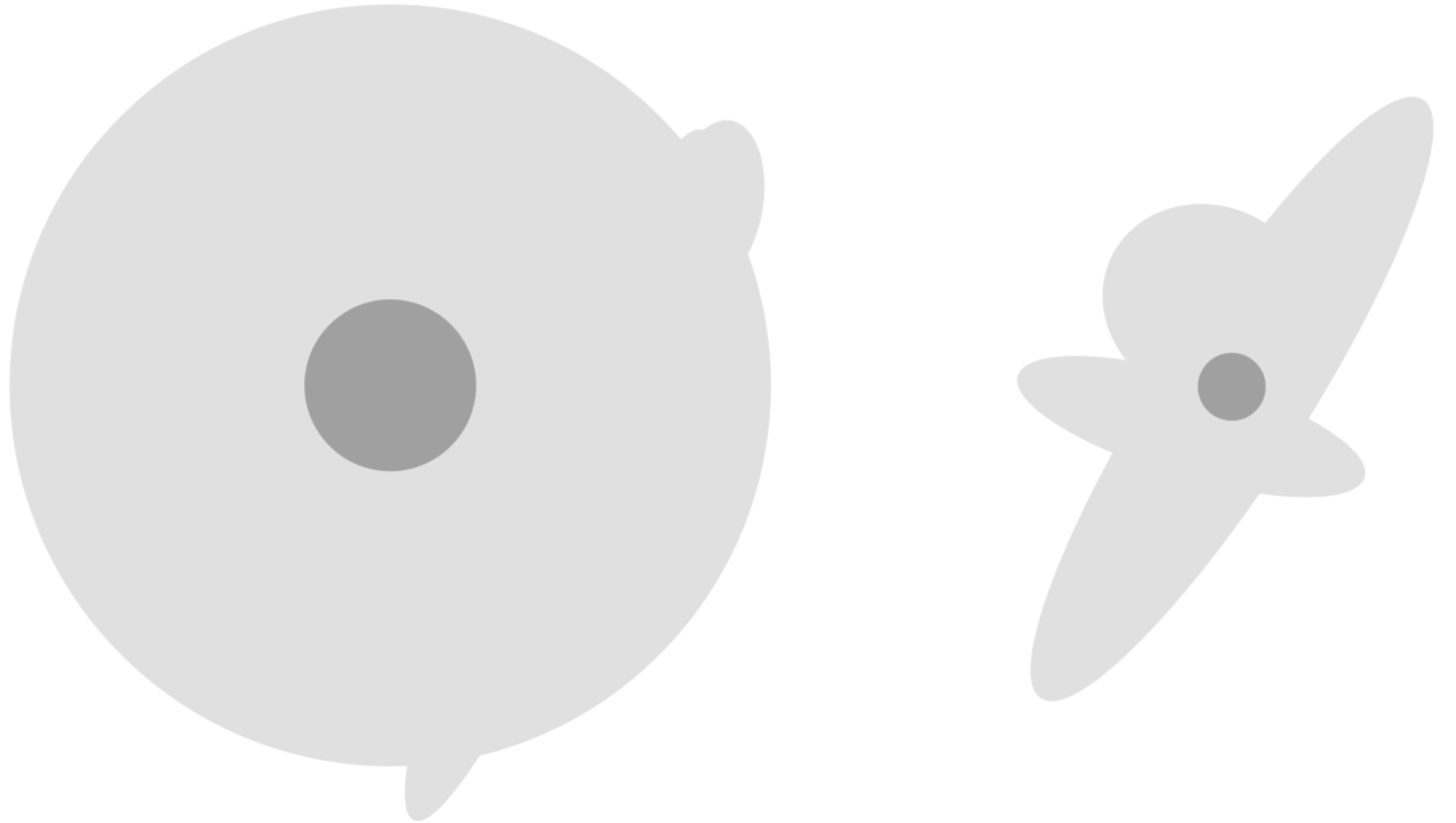}
		\put (22,30) {$E_{k_m+1}$}
		\put (5,40) {$D_{k_m+1}$}
		\put (80,46) {$D_{k_m+1}'$}
		\put (80,34) {$E_{k_m+1}'$}
		\put (60,35) {$\longrightarrow$}
		\put (62,38) {$f$}
	\end{overpic}
	\caption{Top figure: $B_{i_1}'\cap B_j'\neq \emptyset$, so $\diam(A_{i_1}')\simeq_M \diam (A_j')$. On the other hand, $B_{i_1}$ need not intersect $B_j$ and $\diam(A_j)$ might be much smaller than $\diam(A_{i_1})$. Bottom figure: Formation of $D_{k_m+1}$ by taking the union of $A_{i_1}$ with sets $A_j$ such that $B_{i_1}'\cap B_j'\neq \emptyset$.}\label{figure:eggyolk_cluster}
\end{figure}

We continue in the same way, by picking $i_2\in \mathcal F_{m+1}\setminus \{i_1\}$ such that $A_{i_2}\not\subset D_{k_m+1}$. If no such $i_2$ exists, we define $k_{m+1}=k_m+1$. Note that $B_{i_2}'\cap E_{k_m+1}'=\emptyset$ by the choice of $E_{k_m+1}'$, and $B_{i_2}'\cap E_j'=\emptyset$ for each $j\in \{1,\dots,k_m\}$ by the induction assumption \ref{yolk:i:containment}. We define $E_{k_m+2}'=B_{i_2}'$ and $D_{k_m+2}'$ to be the union of $A_{i_2}'$ with the sets $A_j'$ such that $B_{i_2}'\cap B_j'\neq \emptyset$ and $A_j\not\subset  D_i$ for $i\in \{1,\dots,k_{m}+1\}$. Also, set $D_{k_m+2}=f^{-1}(D_{k_m+2}')$ and $E_{k_m+2}=B_{i_2}$.  In this way we produce  $c(M)$-egg-yolk pairs $(D_{k_m+2}',E_{k_m+2}')$ and $(D_{k_m+2},E_{k_m+2})$ such that $E_{k_m+2}'\cap E_{j}'=\emptyset$ for $j\in \{1,\dots,k_m+1\}$. As before, by construction we have $\{ i\in I: B_i'\cap E_{k_m+2}'\neq \emptyset  \}\subset \{i\in I: A_i\subset D_j\,\, \textrm{for some}\,\, j\in \{1,\dots,k_m+2\} \}.$

We claim that this process will stop after finitely many steps. That is, there exists $k_{m+1}> k_m$ with the property that there is no $i\in \mathcal F_{m+1}\setminus \{i_1,\dots,i_{k_{m+1}-k_m}\}$ such that $A_{i}\not\subset D_{j}$ for each $j\in \{k_m+1,\dots,k_{m+1}\}$.  Indeed, by the uniform continuity of  $f^{-1}$, we have 
\begin{align}\label{yolk:radii_induction}
\inf_{i\in \mathcal F_{m+1}} \diam(A_i')>0.
\end{align}
Each $E_i'$, $i=k_m+1,k_m+2,\dots$, is a ball of radius comparable to $\diam(D_i')$; thus, $\diam(E_i')$ is bounded below away from $0$ by \eqref{yolk:radii_induction}. Moreover, the balls $E_i'$ are disjoint and are contained in the compact space $Y$. This shows that this process will necessarily end after a number $k_{m+1}-k_m$ of steps. We also conclude that if $i\in \mathcal F_{m+1}\setminus \{i_1,\dots,i_{k_{m+1}-k_m}\}$, then $A_{i}\subset D_j$ for some $j\in \{k_m+1,\dots,k_{m+1}\}$; this is also trivially true for $i\in \{i_1,\dots,i_{k_{m+1}-k_m}\}$.

We first verify \ref{yolk:i:inclusion} for the index $m+1$. By the definition of $D_j$,  $j\in \{k_m+1,\dots,k_{m+1}\}$, and the induction assumption \ref{yolk:i:inclusion} it is clear that $\bigcup_{i=1}^{k_{m+1}} D_i \subset \bigcup_{i\in I}A_i$. This explains the last part of \ref{yolk:i:inclusion}. If $i\in \mathcal F_{m+1}$, then by the definition of $\mathcal F_{m+1}$ we have $\diam(A_i)>2^{-m-1}L$. If $2^{-m-1}L<\diam(A_i)\leq 2^{-m}L$,  then either $i\in \mathcal F_{m+1}$, so  $A_{i}\subset D_j$ for some $j\in \{k_m+1,\dots,k_{m+1}\}$, or $A_i\subset D_j$ for some $j\leq k_m$. In combination with the induction assumption, this shows the inclusions in \ref{yolk:i:inclusion}. By construction and the induction assumption, \ref{yolk:i:image} and \ref{yolk:i:disjoint}, and \ref{yolk:i:containment}  are automatically satisfied for the index $m+1$. Thus, the proof of the inductive step is completed.
\end{proof}

\section{Proof of Theorem \ref{theorem:main}}\label{section:proof}

\subsection{Preliminaries}
The \textit{$1$-dimensional Hausdorff measure} $\mathscr H^1(A)$ of a set $A\subset \R^n$ is defined by
$$\mathscr{H}^{1}(A)=\lim_{\delta \to 0} \mathscr{H}_\delta^{1}(A)=\sup_{\delta>0} \mathscr{H}_\delta^{1}(A),$$
where
$$
\mathscr{H}_\delta^{1}(A)=\inf \left\{ \sum_{j=1}^\infty \operatorname{diam}(U_j): A \subset \bigcup_j U_j,\, \operatorname{diam}(U_j)<\delta \right\}.
$$
If $\delta=\infty$, the quantity $\mathscr{H}_\infty^{1}(A)$ is called the \textit{$1$-dimensional Hausdorff content} of $A$ and is an outer measure on subsets of $\R^n$. An elementary fact is that
\begin{align*}
\textrm{$\h^1(A)=0$ if and only if $\h^1_{\infty}(A)=0$. }
\end{align*}
We always have
$$\min \{\h^1(A),\diam(A) \}\geq \h^1_\infty(A)$$
and if the set $A$ is connected, then 
\begin{align*}
\h^1_{\infty}(A)= \diam(A).
\end{align*}
See \cite{BuragoBuragoIvanov:metric}*{Lemma 2.6.1, p.~53} for an argument.

A \textit{path} or \textit{curve} is a continuous function $\gamma\colon I \to \R^n$, where $I\subset \R$ is a compact interval. The \textit{trace} of a path $\gamma$ is the image $\gamma(I)$ and will be denoted by $|\gamma|$. The \textit{endpoints} of a path $\gamma\colon [a,b]\to \R^n$ are the points $\gamma(a),\gamma(b)$.

Let $\Gamma$ be a family of curves in $\R^n$. A Borel function $\rho\colon \R^n \to [0,\infty]$ is \textit{admissible} for the path family $\Gamma$ if $$\int_{\gamma}\rho\, ds\geq 1$$
for all rectifiable paths $\gamma\in \Gamma$. We define the \textit{$n$-modulus} of $\Gamma$ as 
$$\md_n \Gamma = \inf_\rho \int \rho^n,$$
where the infimum is taken over all admissible functions $\rho$ for $\Gamma$. By convention, $\md_n \Gamma = \infty$ if there are no admissible functions for $\Gamma$. Note that unrectifiable paths do not affect modulus. Hence, we will assume that families of $n$-modulus zero appearing in the next considerations contain all unrectifiable paths; for example, see \ref{m:integrable} below. We will use the following standard facts about modulus:
\begin{enumerate}[label=(M\arabic*)]
\item The modulus $\md_n$ is an outer measure in the space of all curves in $\R^n$. In particular, it obeys the monotonicity and countable subadditivity laws. \label{m:outer_measure}
\item If $\Gamma_0$ is a path family with $\md_n\Gamma_0=0$, then the family of paths $\gamma$ that have a subpath in $\Gamma_0$ also has $n$-modulus zero. \label{m:subpath}
\item If $\Omega\subset \R^n$ is an open set and $\rho\colon \Omega\to [0,\infty]$ is a Borel function with $\rho\in L^n_{\loc}(\Omega)$, then there exists a path family $\Gamma_0$  with $\md_n\Gamma_0=0$ such that for each path $\gamma\notin \Gamma_0$ with trace in $\Omega$ we have 
$\int_{\gamma}\rho\, ds<\infty$; here we implicitly assume that if $\gamma\notin \Gamma_0$, then $\gamma$ is rectifiable.\label{m:integrable}
\item If $\rho\colon \R^n\to [0,\infty]$ is a Borel function with $\rho=0$ a.e., then there exists a path family $\Gamma_0$ with $\md_n\Gamma_0=0$ such that for each path $\gamma\notin \Gamma_0$ we have $\int_{\gamma}\rho\, ds=0$.\label{m:zero_ae}
\end{enumerate}
See \cite{Vaisala:quasiconformal}*{Chapter 1, pp.~16--20} and \cite{HeinonenKoskelaShanmugalingamTyson:Sobolev}*{Section 5.2} for more details about modulus and proofs of these facts.

\subsection{Finite distortion implies absolute continuity}
The next theorem is the main technical result leading to the proof of the main theorem, Theorem \ref{theorem:main}. We use the notation $m_n^*$ for the $n$-dimensional outer Lebesgue measure in $\R^n$.

\begin{theorem}\label{theorem:upper_gradient}
Let $\Omega\subset \R^n$ be an open set and $f\colon \Omega\to \R^n$ be a topological embedding. Let $X\subset \Omega$ be a set and suppose that there exists a constant $H\geq 1$ such that for all $x\in X$ we have
\begin{align*}
E_f(x)\leq H.
\end{align*} 
Then there exists a Borel function $\rho_f\colon \Omega\to [0,\infty]$ with the following properties.
\begin{enumerate}[\upshape(i)]
	\item \textup{(Support)} There exists a Borel set $U\subset \Omega$ such that $X\subset U$, $m_n(U)=m_n^*(X)$, $m_n(f(U))=m_n^*(f(X))$, and $\rho_f$ is supported on $U$.\label{theorem:upper_gradient:i} 
	\item \textup{(Upper gradient)} There exists a path family $\Gamma_0$ with $\md_n\Gamma_0=0$ such that for all paths $\gamma\notin \Gamma_0$ with trace in $\Omega$ we have
	$$\h^1_\infty(f(|\gamma|\cap X))  \leq \int_\gamma \rho_f\, ds.$$ \label{theorem:upper_gradient:ii}
	\item \textup{(Quasiconformality)} For every Borel set $V\subset \Omega$ we have
	\begin{align*}
\int_V \rho_f^n \leq C(n,H) m_n(f(U\cap V)).
\end{align*}\label{theorem:upper_gradient:iii}
\end{enumerate}
\end{theorem}

\begin{proof} We split the proof into several parts for the convenience of the reader.

\smallskip
\noindent
\textbf{Reduction to a connected domain.} First, we reduce to the case that $\Omega$ is connected. Suppose that $\Omega$ is disconnected and that the theorem is true in each connected component $\Omega_j$, $j\in J$, of $\Omega$. That is, there exists a Borel function $\rho_f$ on $\Omega$ satisfying \ref{theorem:upper_gradient:i}--\ref{theorem:upper_gradient:iii} in each $\Omega_j$.  We verify that these properties hold in all of $\Omega$. By \ref{theorem:upper_gradient:i}, for each $j\in J$, there exists a Borel set $U_j\subset \Omega_j$ such that $X\cap \Omega_j\subset U_j$, $m_n(U_j)=m_n^*(X\cap \Omega_j)$, $m_n(f(U_j))=m_n^*(f(X\cap \Omega_j))$, and $\rho_f|_{\Omega_j}$ is supported on $U_j$. We set $U=\bigcup_{j\in J}U_j$ and observe that 
$$m_n^*(X)=\sum_{j\in J}m_n^*(X\cap \Omega_j)=\sum_{j\in J}m_n(U_j)=m_n(U)$$
and similar equalities hold for $m_n^*(f(X))$ and $m_n(f(U))$. This verifies \ref{theorem:upper_gradient:i}. By \ref{theorem:upper_gradient:ii}, for each $j\in J$, there exists a curve family $\Gamma_j$ of $n$-modulus zero such that for all paths $\gamma\notin \Gamma_j$ with trace in $\Omega_j$ we have
$$\h^1_\infty(f(|\gamma|\cap X))  \leq \int_\gamma \rho_f\cdot \x_{\Omega_j}\, ds.$$
We let $\Gamma_0=\bigcup_{j\in J}\Gamma_j$, which is a family of $n$-modulus zero by the subadditivity of modulus. Then the inequality in \ref{theorem:upper_gradient:ii} is true for all curves $\gamma$ in $\Omega$ that are outside $\Gamma_0$. Finally, \ref{theorem:upper_gradient:iii} is an elementary consequence of the countable additivity of $m_n$.

\smallskip
\noindent
\textbf{Construction of approximate gradients.} From now on, we assume that $\Omega$ is connected. Let $\{V_k\}_{k\in \N}$ be an exhaustion of $\Omega$ by connected open sets such that $\br{V_k}\subset\subset  V_{k+1} \subset \Omega$ for each $k\in \N$. We write $X= \bigcup_{k=1}^\infty X_k$, where $X_k=X\cap V_k$, $k\in \N$. Consider a sequence of open sets $U_{k+1}\subset U_{k}\subset \Omega $, $k\in \N$, such that $X\subset U\coloneqq\bigcap_{k=1}^\infty U_k$, $m_n^*(X)=m_n(U)$, and $m_n^*(f(X))=m_n^*(f(U))$.  

We fix $k\in \N$. Since $E_f\leq H$ on $X$, for each $x\in X_{k}$ there exists an arbitrarily small open set $A_x\subset U_k\cap V_k$ containing $x$ such that
\begin{align*}
&E(A_x) < 2H\,\,\, \textrm{and} \,\,\,  E(f(A_x))<2H;
\end{align*}
recall Definition \ref{definition:eccentric}. These conditions imply that there exists an open ball $B_x$ such that $B_x\subset 2B_x\subset A_x\subset 4HB_x$ and an open ball $B_x'$ such that $B_x'\subset 2B_x'\subset A_x'\subset 4HB_x'$. By considering a smaller set $A_x$, we may also require that
\begin{align*}
c_2(H)B_x \subset U_k\cap V_k\,\,\,  \textrm{and}\,\,\,  \diam(B_x)< c_2(H)^{-1}k^{-1}
\end{align*}
where $c_2(H)$ is a positive constant, to be specified. Thus, $\{(A_x,B_x)\}_{x\in X_{k}}$ and $\{(A_x',B_x')\}_{x\in  X_{k}}$ are families of $(4H)$-egg-yolk pairs (recall the definition from Section \ref{section:egg}) in the compact, connected sets $\br {V_k}$ and $f( \br {V_k})$, respectively.

By the egg-yolk covering lemma, Lemma \ref{lemma:eggyolk},  there exist families $\{(A_i,B_i)\}_{i\in I}$ and $\{(A_i',B_i')\}_{i\in I}$ of $c_1(H)$-egg-yolk pairs in $\br{V_k}$ and $f(\br{V_k})$, respectively, such that $f(A_i)=A_i'$ for each $i\in I$, $X_{k}\subset \bigcup_{i\in I}A_i \subset U_k\cap V_k$, and the families $\{B_i\}_{i\in I}$ and $\{B_i'\}_{i\in I}$ are disjointed. Moreover, $\{B_i\}_{i\in I}$ is a subcollection of $\{B_x\}_{x\in X_{k}}$. We now choose $c_2(H)=c_1(H)+1$, so
\begin{align}\label{theorem:upper_gradient:b_inequalities}
(c_1(H)+1)B_i \subset U_k\cap V_k\,\,\,  \textrm{and}\,\,\,  (c_1(H)+1)\diam(B_i)< k^{-1}.
\end{align}
We note that
\begin{align}\label{theorem:homeo:inclusions}
	A_i \subset c_1(H)B_i \quad \textrm{and} \quad A_i'\subset c_1(H)B_i'.
\end{align}
In addition, since $B_i'$ is a ball, we have
\begin{align}\label{theorem:homeo:diam}
\diam(A_i')^n \leq  c_1(H)^n \diam(B_i')^n \lesssim_{n,H} m_n( B_i').
\end{align}

We set $r_i$ to be the radius of the ball $B_i$, $i\in I$. Consider the function 
$$\rho_k= \sum_{i\in I} \frac{\diam(A_i')}{r_{i}} \x_{(c_1(H)+1)B_i}.$$
By \eqref{theorem:upper_gradient:b_inequalities}, we see that $\rho_k$ is supported on $U_k$. Note that if $A_i\cap |\gamma|\neq \emptyset$ for some rectifiable curve $\gamma$ with $\diam(|\gamma|)>1/k$, then  by \eqref{theorem:homeo:inclusions} we have
$$\int_{\gamma} \x_{(c_1(H)+1)B_i}\, ds \geq r_{i},$$
provided that $|\gamma|$ is not contained in $(c_1(H)+1)B_i$, which is guaranteed by \eqref{theorem:upper_gradient:b_inequalities}. In addition, if $K\subset \Omega$ is a compact set and $A_i\cap K\neq \emptyset$, then by \eqref{theorem:upper_gradient:b_inequalities}  the set $(c_1(H)+1)B_i$ is contained in the open $(1/k)$-neighborhood of $K$, which we denote by $N_{1/k}(K)$. Therefore,
$$ \int_{\gamma}\x_{(c_1(H)+1)B_i} \x_{N_{1/k}(K)}\, ds \geq r_i.$$
We conclude that for each compact set $K\subset \Omega$ and every rectifiable curve $\gamma$ with $\diam(|\gamma|)>1/k$ we have
\begin{align}\label{theorem:upper_gradient:rho_k}
\mathscr H^1_{\infty}( f(|\gamma|\cap X_k\cap K))\leq \sum_{\substack{i:A_i\cap |\gamma| \neq \emptyset \\ A_i\cap K\neq \emptyset}}\diam(A_i') \leq  \int_{\gamma} \rho_k \x_{N_{1/k}(K)}\, ds.
\end{align}

By Lemma \ref{lemma:bojarski}, the fact that $\{B_i\}_{i\in I}$ is disjointed,  \eqref{theorem:homeo:diam}, and the fact that $\{B_i'\}_{i\in I}$ is disjointed, we have
\begin{align*}
\int \rho_k^n &\lesssim_{n,H} \int \left( \sum_{i\in I} \frac{\diam(A_i')}{r_{i}} \x_{B_i} \right)^n\simeq_{n,H}  \int \sum_{i\in I} \frac{\diam(A_i')^n}{r_{i}^n} \x_{B_i}\\
&\simeq_{n,H}  \sum_{i\in I} \diam(A_i')^n \lesssim_{n,H}\sum_{i\in I} m_n(B_i') \lesssim_{n,H}  m_n(f(U_k)).
\end{align*}
Moreover, for each compact set $K\subset \Omega$, the same computation shows that
\begin{align}\label{theorem:quasiconformal:rho_k}
\int_K \rho_k^n \lesssim_{n,H} m_n(f(U_k \cap N_{1/k} (K))).
\end{align}
Observe that the latter measure is finite for large $k\in \N$ and bounded as $k\to \infty$. 

\smallskip
\noindent
\textbf{Compactness argument.}
The uniform upper bounds of \eqref{theorem:quasiconformal:rho_k}, combined with the Banach--Alaoglu theorem \cite{HeinonenKoskelaShanmugalingamTyson:Sobolev}*{Theorem 2.4.1} and a diagonal argument imply that there exists a Borel function $\rho_f\colon \Omega\to [0,\infty]$ with $\rho_f\in L^n_{\loc}(\Omega)$ and a subsequence of $\rho_k$  that converges to $\rho_f$ weakly in $L^n(K)$ for each compact set $K\subset \Omega$; see \cite{HeinonenKoskelaShanmugalingamTyson:Sobolev}*{Lemma 3.3.19} for a variant of this statement. For simplicity, we denote the subsequence by $\rho_k$, $k\in \N$.  

The fact that $U_{k+1}\subset U_k$, $k\in \N$, implies that $\rho_k$ is supported on $U_k$, $k\in \N$.  Passing to the weak limit, we conclude that $\rho_f$ is supported on $U=\bigcap_{k=1}^\infty U_k$, as required in \ref{theorem:upper_gradient:i}. Let $V\subset \Omega$ be a Borel set and $K\subset V$ be a compact set. By the weak convergence of $\rho_k$ to $\rho_f$ in $L^n(K)$ and  \eqref{theorem:quasiconformal:rho_k} we have
\begin{align*}
\int_K \rho_f^n \leq \liminf_{k\to\infty}\int_K \rho_k^n \lesssim_{n,H} m_n(f(U \cap K)) \lesssim_{n,H} m_n(f(U\cap V)).
\end{align*}
The inner regularity of Lebesgue measure  completes the proof of \ref{theorem:upper_gradient:iii}.

Finally, we show \ref{theorem:upper_gradient:ii}. By Mazur's lemma \cite{HeinonenKoskelaShanmugalingamTyson:Sobolev}*{Section 2.3}, for each compact set $K\subset \Omega$ there exist convex combinations $\widetilde \rho_k$ of $\rho_k,\rho_{k+1},\dots,\rho_{m(k)}$, where $m(k)\geq k$, such that  $\widetilde \rho_k$ converges strongly to $\rho_f$ in $L^n(V)$ for some neighborhood $V$ of $K$. By \eqref{theorem:upper_gradient:rho_k} and the fact that $X_k\subset X_{k+1}$, $k\in \N$, we have
\begin{align*}
\h^1_{\infty}(f(|\gamma|\cap X_{k}\cap K )) \leq \int_{\gamma}\widetilde \rho_k \x_{N_{1/k} (K)}\, ds
\end{align*}
whenever $\diam(|\gamma|)>1/k$. Moreover, $X\cap K\subset X_k$ for all sufficiently large $k\in \N$, so
$$\h^1_{\infty}(f(|\gamma|\cap X  \cap K)) \leq \int_{\gamma} \widetilde \rho_k \x_{N_{1/k}(K)}\, ds$$
whenever $\diam(|\gamma|)>1/k$. Observe that $\widetilde \rho_k \x_{N_{1/k}(K)}$ converges to $\rho_f\x_K$ in $L^n(\R^n)$. By Fuglede's lemma \cite{Vaisala:quasiconformal}*{Theorem 28.1} there exists a path family $\Gamma(K)$ of $n$-modulus zero such that for all paths $\gamma\notin \Gamma(K)$ we have
$$\int_{\gamma} \widetilde \rho_k \x_{N_{1/k}(K)}\, ds \to \int_{\gamma}\rho_f \x_{K}\, ds$$
as $k\to\infty$. Given a non-constant path $\gamma\notin \Gamma(K)$, we then have
$$\h^1_{\infty}(f(|\gamma|\cap X  \cap K)) \leq \int_{\gamma}\rho_f\, ds.$$
Let $\Gamma_0= \bigcup_{k=1}^\infty \Gamma(\br {V_k})$, which is a family of $n$-modulus zero.  If $\gamma\notin \Gamma_0$ is a non-constant path with trace in $\Omega$, then there exists $ k\in \N$ such that $|\gamma|\subset \br{V_k}$, so
$$\h^1_{\infty}(f(|\gamma|\cap X  )) = \h^1_{\infty}(f(|\gamma|\cap X \cap \br{V_k} ))\leq  \int_{\gamma}\rho_f\, ds.$$
Constant paths satisfy as well this inequality trivially.
\end{proof}

\subsection{Completing the proof of Theorem \ref{theorem:main}}
The following statement is a consequence of Theorem \ref{theorem:upper_gradient}.

\begin{corollary}\label{corollary:upper_gradient}
Let $\Omega\subset \R^n$ be an open set and $f\colon \Omega\to \R^n$ be a topological embedding. Let $G\subset \Omega$ be a set such that for all $x\in G$ we have
\begin{align*}
E_f(x)<\infty
\end{align*} 
and either $m_n(G)=0$ or $m_n(f(G))=0$. Then there exists a path family $\Gamma_0$ with $\md_n \Gamma_0=0$ such that for all paths $\gamma\notin\Gamma_0$ with trace in $\Omega$ we have 
	$$\mathscr H^1( f(|\gamma|\cap G))=0.$$
\end{corollary}

\begin{proof}
We write $G=\bigcup_{k=1}^\infty G_k$, where $E_f(x)\leq k$ for $x\in G_k$. We fix $k\in \N$ and consider the function $\rho_k$ given by Theorem \ref{theorem:upper_gradient} and corresponding to the set $X=G_k$. If $m_n(G)=0$, then $\rho_k=0$ a.e.\ by part \ref{theorem:upper_gradient:i}. If $m_n(f(G))=0$, then $\rho_k=0$ a.e.\ by part \ref{theorem:upper_gradient:iii}. In both cases, $\rho_k=0$ a.e.  By property \ref{m:zero_ae}, this implies that there exists a path family $\Gamma_k$ of $n$-modulus zero such that for $\gamma\notin \Gamma_k$ we have
$$\int_{\gamma}\rho_k\, ds=0.$$
Combining this with part \ref{theorem:upper_gradient:ii} of Theorem \ref{theorem:upper_gradient}, we see that there exists a path family $\Gamma_k'$ with $n$-modulus zero such that 
$$\mathscr H^1( f(|\gamma|\cap G_k))=\mathscr H^1_{\infty}( f(|\gamma|\cap G_k))=0$$
for all paths $\gamma\notin \Gamma_k'$ with trace in $\Omega$.  The desired path family is $\Gamma_0=\bigcup_{k=1}^\infty \Gamma_k'$.
\end{proof}

With the aid of Corollary \ref{corollary:upper_gradient} one can immediately deduce Theorem \ref{theorem:main} from the following slightly more general statement.

\begin{theorem}\label{theorem:main:generalization}
Let $\Omega\subset \R^n$ be an open set and $f\colon \Omega \to \R^n$ be a topological embedding. Let $A,G\subset \Omega$ be  sets such that $A\in \CNED$ and assume that there exists a path family $\Gamma_0$ with $\md_n \Gamma_0=0$ such that for all paths $\gamma\notin\Gamma_0$ with trace in $\Omega$ we have 
	$$\mathscr H^1( f(|\gamma|\cap G))=0.$$
Suppose that there exists a constant $H\geq 1$ such that for all $x\in \Omega\setminus (A\cup G)$ we have
\begin{align*}
 E_f(x) \leq H,
\end{align*}
and for all $x\in G$ we have
\begin{align*}
 E_f(x)<\infty.
\end{align*}
Then $f$ is $K$-quasiconformal in $\Omega$, for some $K\geq 1$ depending only on $n$ and $H$.
\end{theorem}

We finally focus on proving Theorem \ref{theorem:main:generalization}. We require the next lemma on maps that are absolutely continuous along paths. 
 
\begin{lemma}[\cite{Vaisala:quasiconformal}*{Theorem 5.3}]\label{lemma:path_change_variables}
Let $\Omega\subset \R^n$ be an open set, $f\colon \Omega\to \R^n$ be a continuous map, $\rho_f\colon \Omega\to [0,\infty]$ be a Borel function, and $\gamma\colon [a,b]\to \Omega$ be a rectifiable path. Suppose that for every interval $[s,t]\subset [a,b]$ we have
$$|f(\gamma(t))-f(\gamma(s))| \leq \int_{\gamma|_{[s,t]}} \rho_f\, ds<\infty.$$
Then for every Borel function $\rho\colon f(\Omega)\to [0,\infty]$ we have 
$$\int_{f\circ \gamma}\rho\, ds \leq \int_{\gamma} (\rho \circ f) \cdot \rho_f\, ds.$$
\end{lemma}

A \textit{topological ring} $R$ is a bounded open set in $\R^n$ whose boundary has two components, say $F_1$ and $F_2$. For a topological ring $R$ we denote by  $\Gamma(R)$ the family of curves joining $F_1$ and $F_2$ in $R$; that is, the curves of $\Gamma(R)$ are contained in $R$, except for the endpoints, which lie in different components of $\partial R$.  We will use the fact that if $\mathcal F$ is a family of curves that is closed under subpaths, then 
\begin{align*}
\md_n(\Gamma(R)\cap \mathcal F)= \md_n(\Gamma(F_1,F_2;\R^n)\cap \mathcal F).
\end{align*}
In particular, this is true if $\mathcal F=\mathcal F_{\sigma}(A)$ for some set $A$; recall that $\mathcal F_{\sigma}(A)$ is the family of curves meeting $A$ at countably many points. Hence, if $A\in \CNED$, 
\begin{align}\label{equality:modulus_cned}
\begin{aligned}
\md_n(\Gamma(R)\cap \mathcal F_{\sigma}(A))&=  \md_n(\Gamma(F_1,F_2;\R^n)\cap \mathcal F_{\sigma}(A))\\ &= \md_n\Gamma(F_1,F_2;\R^n) = \md_n\Gamma(R).
\end{aligned}
\end{align}

In order to show that the map $f$ of Theorem \ref{theorem:main:generalization} is quasiconformal, we will use the ring definition of quasiconformality as stated in the next theorem. 

\begin{theorem}[\cite{Vaisala:quasiconformal}*{Theorem 36.1}]\label{theorem:ring_definition}
Let $\Omega\subset \R^n$ be an open set and $f\colon \Omega\to \R^n$ be a topological embedding. If there exists $K\geq 1$ such that for each topological ring $R\subset \br R\subset\subset \Omega$ we have $\md \Gamma(R) \leq K \md \Gamma(f(R))$, then $f$ is $K$-quasiconformal. 
\end{theorem}

\begin{proof}[Proof of Theorem \ref{theorem:main:generalization}]
We apply Theorem \ref{theorem:upper_gradient} with $X=\Omega\setminus (A\cup G)$. Denote by $\Gamma_0'$ the union of the exceptional path families given by Theorem \ref{theorem:upper_gradient} \ref{theorem:upper_gradient:ii} and by the statement of Theorem \ref{theorem:main:generalization}, and note that $\md_n\Gamma_0'=0$.  By Theorem \ref{theorem:upper_gradient}, there exists a Borel function $\rho_f\colon \Omega\to [0,\infty]$ with $\rho_f\in L^n_{\loc}(\Omega)$ such that for all paths $\gamma\notin \Gamma_0'$ with trace in $\Omega$ we have 
\begin{align}\label{proof:upper_gradient}
\h^1_{\infty}(f(|\gamma|\setminus A)) =\h^1_{\infty}(f(|\gamma|\cap (\Omega\setminus (A\cup G))))  \leq \int_{\gamma}\rho_f\, ds
\end{align}
and 
$$\int_V \rho_f^n \leq C(n,H) m_n(f(V))$$
for each Borel set $V\subset \Omega$. The latter implies that for every Borel function $\rho\colon f(\Omega)\to [0,\infty]$ we have
\begin{align}\label{proof:quasiconformality}
\int (\rho\circ f)\cdot  \rho_f^n \leq C(n,H) \int \rho.
\end{align}
By enlarging the exceptional family $\Gamma_0'$, still requiring that $\md_n\Gamma_0'=0$, we may assume that it has the additional properties that
$$\int_{\gamma}\rho_f\, ds<\infty$$
for each $\gamma\notin \Gamma_0'$ with trace in $\Omega$ and that if $\gamma\notin \Gamma_0'$ then all subpaths of $\gamma$ have the same property; see properties \ref{m:subpath} and \ref{m:integrable}. 

Let $R\subset \br R\subset \subset  \Omega$ be a topological ring. Let $\gamma\colon[a,b]\to \Omega$ be a path in $(\Gamma(R)\cap \mathcal F_{\sigma}(A))\setminus \Gamma_0'$. By \eqref{proof:upper_gradient}  and the fact that every subpath of $\gamma$ lies outside $\Gamma_0'$ we have
\begin{align*}
|f(\gamma(t))-f(\gamma(s))| &\leq \diam(f(\gamma([s,t]))) = \h^1_{\infty}(f(\gamma([s,t]))) \\
&=\h^1_{\infty}(f(\gamma([s,t])\setminus A)) \leq \int_{\gamma|_{[s,t]}}\rho_f\, ds<\infty
\end{align*}
for every interval $[s,t]\subset [a,b]$. Let $\rho\colon f(\Omega)\to [0,\infty]$ be admissible for $\Gamma(f(R))$. Then by Lemma \ref{lemma:path_change_variables} we have
$$ \int_{\gamma}(\rho\circ f) \cdot \rho_f\, ds\geq \int_{f\circ \gamma}\rho \, ds \geq 1. $$
This shows that $(\rho\circ f)\cdot \rho_f$ is admissible for $(\Gamma(R)\cap \mathcal F_{\sigma}(A))\setminus \Gamma_0'$, so using \eqref{proof:quasiconformality}, we arrive at
\begin{align*}
\md_n (\Gamma(R)\cap \mathcal F_{\sigma}(A))&=\md_n ((\Gamma(R)\cap \mathcal F_{\sigma}(A))\setminus \Gamma_0')\\
&\leq \int (\rho\circ f)^n \cdot \rho_f^n \leq C(n,H) \int \rho^n. 
\end{align*} 
We conclude that
\begin{align*}
\md_n (\Gamma(R)\cap \mathcal F_{\sigma}(A))\leq C(n,H)\md_n \Gamma(f(R)). 
\end{align*}
Finally, the assumption that $A\in \CNED$ and \eqref{equality:modulus_cned} imply that
$$\md_n (\Gamma(R)\cap \mathcal F_{\sigma}(A))= \md_n \Gamma(R).$$
An application of Theorem \ref{theorem:ring_definition} completes the proof. 
\end{proof}

\subsection*{Data availability}
The paper has no associated data.
\subsection*{Conflict of interest}
The author states that there is no conflict of interest. 

\bibliography{biblio}
\end{document}